\documentclass[12pt,a4paper,oneside]{article}
\usepackage[a4paper,left=3cm,right=3cm, top=3cm, bottom=3cm]{geometry}
\usepackage[latin1]{inputenc}
\usepackage{enumerate}
\usepackage{mathrsfs}
\usepackage[T1]{fontenc}
\usepackage {graphicx}
\usepackage{amsmath}
\usepackage{amsthm}
\usepackage{amssymb}
\usepackage{hyperref}
\theoremstyle{plain}
\newtheorem{thm}{Theorem}[section]  
\newtheorem{cor}[thm]{Corollary}
\newtheorem{lem}[thm]{Lemma}
\newtheorem{prop}[thm]{Proposition}
\theoremstyle{definition}
\newtheorem{defn}[thm]{Definition}

\theoremstyle{remark}
\newtheorem{remark}[thm]{Remark}
\theoremstyle{problem}
\newtheorem{problem}[thm]{Problem}

\newcommand{\R}{\mathbb{R}}

\newcommand{\U}{\mathcal{U}}
\newcommand{\G}{\mathcal{G}}

\DeclareMathOperator{\supp}{supp}
\DeclareMathOperator{\ims}{ims}
\DeclareMathOperator{\Lev}{Lev}
\DeclareMathOperator{\spanne}{span}
\DeclareMathOperator{\htte}{ht}
\DeclareMathOperator{\cf}{cf}
\DeclareMathOperator{\Clop}{Clop}

\newcommand{\vertiii}[1]{{\left\vert\kern-0.25ex\left\vert\kern-0.25ex\left\vert #1 
    \right\vert\kern-0.25ex\right\vert\kern-0.25ex\right\vert}}
\title{On compact trees with the coarse wedge topology}
\author{Jacopo Somaglia\footnote{Research was supported in part by the 
Universit\`a degli Studi of Milano (Italy), in part by the Gruppo Nazionale per l'Analisi Matematica, la Probabilit\`a e le loro Applicazioni (GNAMPA) of the Istituto Nazionale di Alta Matematica (INdAM) of Italy and in part by the research grant GA\v{C}R 17-00941S.}}
\date{}

\begin{document}
\maketitle

\begin{abstract}
\noindent
In the present paper we investigate the class of compact trees, endowed with the coarse wedge topology, in the area of non-separable Banach spaces. We describe Valdivia compact trees in terms of inner structures and we characterize the space of continuous functions on them. Moreover we prove that the space of continuous functions on an arbitrary tree with height less than $\omega_1\cdot\omega_0$ is Plichko.
\end{abstract}

\section{Introduction}

A \textit{tree} is a partially ordered set $(T,<)$ such that the set of predecessors $\{s\in T:s<t\}$ of any $t\in T$ is well-ordered by $<$. There are several natural topologies that can be defined by using the order structures of trees, see \cite{Nyikos2}. Among them, the coarse wedge topology is a topology for which the tree $T$ is a compact Hausdorff space whenever $T$ satisfies certain structural properties.\\
In the present paper we investigate the relations between the coarse wedge topology and both classes of Valdivia compacta and Plichko Banach spaces. Plichko spaces are a wide class of Banach spaces that extend the class of weakly Lindel\"{o}f determined (WLD) Banach spaces. It was introduced in \cite{Plichko1} and was studied under equivalent definitions in \cite{DeviGode}, \cite{Vald1} and \cite{Vald2}; we refer to \cite{Kalenda2} for a detailed survey in this area. Plichko spaces and the related class of compact spaces, the so called Valdivia compacta, appear in many different areas, see \cite{Kalenda3} for the details and  \cite{CuthFab}, \cite{BohaHamKal}, \cite{CorreaTausk} for some recent results in Banach spaces, $C^*$-algebras and topology.\\
W. Kubi\'{s} introduced the concept of projectional skeletons in \cite{Kubis1}, where he adapted the definition of retractional skeleton (see \cite{KubisMicha}, \cite{Cuth1}) from the topological setting to Banach spaces. Roughly speaking, a projectional skeleton decomposes the Banach space into smaller separable subspaces, see \cite{Cuth2}, \cite{CuthFab}, \cite{Kalenda8}. Banach spaces (compact spaces) with a projectionl skeleton (retractional skeleton, respectively) are the non-commutative counterpart of the aforementioned classes, in the sense that a Banach space (compact space) is Plichko (Valdivia) if and only if it admits a commutative projectional skeleton (commutative  retractional skeleton, respectively). Although Plichko spaces and Banach spaces with a projectional skeleton, as well as Valdivia compacta and compact spaces with a retractional skeleton, share many structural and topological properties, they do not coincide. A simple example of non Valdivia compact space with rectractional skeleton is the compact ordinal interval $[0,\omega_2]$, \cite{KubisMicha}. It is more difficult to prove that the Banach space $C([0,\omega_2])$, which has a projectional skeleton, is not Plichko \cite{Kalenda7}.\\
In \cite{SomaTree} we studied the class of trees endowed with the coarse wedge topology, providing new examples of non Valdivia compact spaces with retractional skeletons. In the same paper it was proved that every tree with height less or equal than $\omega_1 +1$ is Valdivia and no Valdivia tree has height greater than $\omega_2$. Moreover, an example of non Valdivia tree with height $\omega_1 +2$ was given. In the present paper we follow the same research line. In particular we investigate the space of continuous functions on compact trees. We prove that $C(T)$ is Plichko whenever the height of $T$ is less than $\omega_1\cdot \omega_0$. Finally we extend Theorem 4.1 of \cite{SomaTree}, characterizing Valdivia compact trees with height less than $\omega_2$. It turns out that this characterization depends only on the behaviour of the tree on the levels with uncountable cofinality.\\
We now outline how the paper is organized. In the remaining part of the introductory section notation and basic notions addressed in this paper are given. Section 2 contains details of notation, basic definitions and some preliminary results on trees. Section 3 is devoted to characterizing Valdivia compact trees with height less than $\omega_2$. Section 4 deals with the class of continuous functions on a compact tree. It is shown that, if $C(T)$ is $1$-Plichko, then $T$ is Valdivia. We also prove that, if $T$ is an arbitrary tree with height less than $\omega_1\cdot \omega_0$, then $C(T)$ is a Plichko space.\\
We denote with $\omega_0$ the set of natural numbers (including 0) with the usual order. Given a set $X$ we denote by $|X|$ the cardinality of the set $X$, by $[X]^{\leq \omega_0}$ the family of all countable subsets of $X$.\\
All the topological spaces are assumed to be Hausdorff and completely regular. Given a topological space $X$ we denote by $\overline{A}$ the closure of $A\subset X$. We say that $A\subset X$ is countably closed if $\overline{C}\subset A$ for every $C\in [A]^{\leq \omega_0}$.\\
Given a topological compact space (loc. compact space) $K$ we use $C(K)$ ($C_0(K)$) to denote the space of all real-valued or complex-valued continuous functions on $K$ (all real-valued or complex-valued continuous functions on $K$ vanishing at infinity) with the usual norm. By the Riesz representation theorem the elements of $C(K)^{*}$ are considered as measures. If $\mu\in C(K)^*$, we denote by $\|\mu\|$ its norm. If $\mu$ is a non-negative measure, we denote by supp$(\mu)$ the support of the measure $\mu$, i.e. the set of those points $x\in K$ such that each neighborhood of $x$ has positive $\mu$-measure. The support of a measure $\mu\in C(K)^*$ coincides with the support of its total variation $|\mu|$.\\
Given a Banach space $X$ and a subset $A\subset X$ we denote by $\spanne(A)$ the linear hull of $A$. $B_{X}$ is the norm-closed unit ball of $X$ (i.e. the set $\{x\in X:\, \|x\|\leq 1\}$). As usual $X^{*}$ stands for the (topological) dual space of $X$. A set $D\subset X^{*}$ is said $\lambda$-norming if 
\begin{equation*}
\|x\|\leq \lambda \sup\{|x^{*}(x)|:\,x^{*}\in D\cap B_{X^{*}}\}
\end{equation*}
for every $x\in X$. We say that a set $D\subset X^{*}$ is \textit{norming} if it is \textit{$\lambda$-norming} for some $\lambda\geq 1$. A subspace $S\subset X^*$ is called a \textit{$\Sigma$-subspace} of $X^*$ if there is a set $M\subset X$ such that $\overline{\spanne}(M)=X$ and that
\begin{equation*}
S=\{f\in X^*:\{m\in M:f(m)\neq 0\} \mbox{ is countable}\}.
\end{equation*}
A Banach space $X$ is called \textit{Plichko} (\textit{$\lambda$-Plichko}) \textit{space} if $X^*$ admits a norming ($\lambda$-norming, respectively) $\Sigma$-subspace. Let $\Gamma$ be an arbitrary set, we put
\begin{equation*}
\Sigma(\Gamma)=\{x\in \R^{\Gamma}:|\{\gamma\in \Gamma:x(\gamma)\neq 0\}|\leq \omega_0\}.
\end{equation*}
Let $K$ be a compact space, we say that $A\subset K$ is a $\Sigma$-\textit{subset} of $K$ if there exists a homeomorphic injection $h$ of $K$ into some $\R^{\Gamma}$ such that $h(A)=h(K)\cap\Sigma(\Gamma)$. $K$ is called \textit{Valdivia compact space} if $K$ has a dense $\Sigma$-subset. 

\section{Basic notions on trees}

We recall that a \textit{tree} is a partially ordered set $(T,<)$ such that the set of predecessors $\{s\in T:s<t\}$ of any $t\in T$ is well-ordered by $<$. A tree $T$ is said to be \textit{rooted} if it has only one minimal element, called \textit{root}. A totally ordered subset of $T$ is called a \textit{chain} of the tree $T$. A maximal chain is said \textit{branch}. $T$ is called \textit{chain complete} if every chain has a supremum. For any element $t\in T$, $\htte(t,T)$ denotes the order type of $\{s\in T:s<t\}$. For any ordinal $\alpha$, the set $\Lev_{\alpha}(T)=\{t\in T:\htte(t,T)=\alpha\}$ is called the $\alpha$\textit{th level} of $T$. The height of $T$ is denoted by $\htte(T)$, and it is the least $\alpha$ such that $\Lev_{\alpha}(T)=\emptyset$. For an element $t\in T$, $\cf(t)$ denotes the cofinality of $\htte(t,T)$, where $\cf(t)=0$ when $\htte(t,T)$ is a successor ordinal, and $\ims(t)=\{s\in T:t<s, \,\htte(s,T)=\htte(t,T)+1\}$ denotes the set of \textit{immediate successors} of $t$. Given a subset $S$ of a tree $T$, and an element $t\in S$, we denote by $\ims_{S}(t)$ the set of immediate successors of $t$ in $S$ with the inherited order. Let $T$ be a tree of height $\alpha$, for each $\beta<\alpha$ we denote by $T_{\beta}=\bigcup_{\gamma\leq \beta}\Lev_{\gamma}(T)$ and by $T_{<\beta}=\bigcup_{\gamma< \beta}\Lev_{\gamma}(T)$.\\
For $t\in T$ we put $V_t=\{s\in T:s\geq t\}$ and $\hat{t}=\{s\in T:s\leq t\}$. In the present work we consider $T$ endowed with the \textit{coarse wedge topology}. We mention that the coarse wedge topology coincides with the path topology on the set of all initial chains of a tree $T$ which is a tree itself when ordered by inclusion, we refer to \cite{Todo1},\cite{Gruen2} for the details.\\
The coarse wedge topology on a tree $T$ is the one whose subbase is the set of all $V_t$ and their complements, where $t$ is  either minimal or on a successor level. If ht$(t,T)$ is a successor or $t$ is the minimal element, a local base at $t$ is formed by sets of the form
\begin{equation*}
W_{t}^{F}=V_t\setminus \bigcup\{V_s:s\in F\},
\end{equation*}
where $F$ is a finite set of immediate successors of $t$. In case ht$(t,T)$ is limit, a local base at $t$ is formed by sets of the form
\begin{equation*}
W_{s}^{F}=V_s\setminus \bigcup\{V_r:r\in F\},
\end{equation*}
where $s<t$, ht$(s,T)$ is a successor and $F$ is a finite set of immediate successor of $t$. We refer to \cite{Todo3} and \cite{Nyikos2} for further information.\\
Since we are interested in compact spaces, we recall that, by \cite[Corollary 3.5]{Nyikos2}, a tree  $T $ is compact Hausdorff in the coarse wedge topology if and only if $T$ is chain complete and has finitely many minimal elements. For this reason, from now on we will consider only chain complete trees with a unique minimal element. In these settings the operation $t\wedge s=\max(\hat{t}\cap\hat{s})$ is well-defined for every $s,t\in T$.\\
Given a subset $S$ of a tree $T$, there are two natural topologies on $S$: the subspace topology and the coarse wedge topology generated by the inherited order. We shall prove that these topologies sometimes coincide.

\begin{lem}\label{Subspacetopology}
Let $S$ be closed subset of a tree $T$. Suppose that $S$ is closed under $\wedge$ (i.e. if $s,t\in S$, then $s\wedge t\in S$). Then the subspace topology coincides with the coarse wedge topology on $S$.
\end{lem}

\begin{proof}
We firstly observe that if $S$ is a branch of $T$, then the two topologies coincide with the interval topology. We shall prove that if $S$ is endowed with the coarse wedge topology, then it is a compact space. We observe that, since $T$ is chain complete, any chain in $S$ has a supremum in $T$. By the closedness of $S$, the supremum belongs to $S$. Moreover, since $S$ is closed under the $\wedge$ operation and $T$ is rooted, we deduce that $S$ is rooted too. Therefore, by \cite[Corollary 3.5]{Nyikos2} $S$, endowed with the coarse wedge topology, is a compact Hausdorff space. We shall prove that the coarse wedge topology on $S$ is coarser than the subspace topology.\\
Let $x\in S$ and suppose that $x$ is on a successor level in $T$. Since $S$ is closed, $x$ is also on a successor level in $S$. Let $W_{x}^{F}\subset S$ be an open basic neighborhood of $x$, where $F=\{t_i\}_{i=1}^{n}\subset \ims_{S}(x)$. For each $t_i\in T$ there exists a unique $u(t_i)\in \ims_{T}(x)$ such that $t_i\geq u(t_i)$. Let $F_1=\{u(t_i)\}_{i=1}^{n}$. Hence we have $W_{x}^{F}\supset W_{x}^{F_1}\cap S$.\\
Let $x\in S$ and suppose that $x$ belongs to a limit level in $T$. Now two cases are possible:
\begin{itemize}
\item $x$ is on a limit level in $S$. Let $W_{s}^{F}\subset S$ be an open basic neighborhood of $x$, where $s<x$ is on a successor level in $S$ and $F=\{t_i\}_{i=1}^{n}\subset \ims_{S}(x)$. Let $F_1=\{u (t_i)\}_{i=1}^{n}$ and $\{s+1\}=\hat{x}\cap \Lev_{\htte(s,T)+1}(T)$. Hence $W_{s+1}^{F_1}$ is an open basic neighborhood of $x$ in $T$ and $W_{s+1}^{F_1}\cap S\subset W_{s}^{F}$.
\item $x$ is on a successor level in $S$. Let $W_{x}^{F}\subset S$ be an open basic neighborhood of $x$, where $F=\{t_i\}_{i=1}^{n}\subset \ims_{S}(x)$. Since $x$ is on a successor level in $S$, $x$ has an unique immediate predecessor, say $x-1$. Since $x$ is on a limit level in $T$ and $S$ is closed in $T$, there exists $s \in T$ on a successor level such that $x-1<s<x$ and $W_{s}^{x}\cap S=\emptyset$. Indeed, let $x-1<s<x$ and suppose that $y\in W_{s}^{x}\cap S$, then we get $x-1<s\leq x\wedge y<x$. Since $x\wedge y\in S$ we find a contradiction because of the maximality of $x-1$. Hence $W_{s}^{F_1}\cap S\subset W_{x}^{F}$, where $F_1=\{u (t_i)\}_{i=1}^{n}$.
\end{itemize}
Since $S$ is a compact Hausdorff space in both topologies, we obtain the assertion.
\end{proof}

As a consequence we obtain the following result.

\begin{cor}\label{Metrizabletree}
Let $C$ be a countable subset of a tree $T$. Then $\overline{C}$ is a metrizable subspace of $T$. 
\end{cor}

\begin{proof}
Let $C$ be a countable subset of $T$. Let $C_{\wedge}$ be the smallest subset of $T$ containing $C$ and closed under the $\wedge$ operation. It is clear that $C_{\wedge}$ is a countable subset. We shall prove that $\overline{C_{\wedge}}$ is closed under $\wedge$. Let $s,t\in \overline{C_{\wedge}}$ and consider $s\wedge t$. Suppose that $s,t$ are incomparable elements, otherwise the assertion would follows immediately. Let $\{u(t)\}=\ims(s\wedge t)\cap \hat{t}$ and $\{u(s)\}=\ims(s\wedge t)\cap \hat{s}$. Since $s,t\in \overline{C_{\wedge}}$, there are $s_1, t_1\in C_{\wedge}$ such that $s_1\in V_{u(s)}$ and $t_1\in V_{u(t)}$. Therefore we have $s\wedge t=s_1\wedge t_1\in C_{\wedge}$.\\
We observe that if $t\in \overline{C_{\wedge}}\setminus C_{\wedge}$, then $t$ belongs to a limit level. Indeed, suppose that $t\in \overline{C_{\wedge}}\setminus C_{\wedge}$ and $t$ belongs to a successor level. Then there exists an infinite set $A$, such that for each $\alpha\in A$ there exists $t_{\alpha}\in \ims(t)$ satisfying $V_{t_{\alpha}}\cap C_{\wedge}\neq \emptyset$ . Pick $s_{\alpha_i}\in V_{t_{\alpha_i}}\cap C_{\wedge}$ such that $\alpha_i\in A$ with $i=1,2$, then we have $s_{\alpha_1}\wedge s_{\alpha_2}=t\in C_{\wedge}$, a contradiction. Therefore, any $t\in \overline{C_{\wedge}}\setminus C_{\wedge}$ belongs to a limit level and so any chain is at most countable. Combining Lemma \ref{Subspacetopology} with \cite[Theorem 2.8]{Nyikos1} we obtain that $\overline{C_{\wedge}}$ is a separable Corson compact space, hence metrizable. Since $\overline{C}\subset \overline{C_{\wedge}}$, we obtain the assertion.
\end{proof}

Let $X$ be a topological space, a family $\mathcal{U}$ of subsets of $X$ is $T_0$-separating in $X$ if for every two distinct elements $x,y\in X$ there is $U\in\mathcal{U}$ satisfying $|\{x,y\}\cap U|=1$. A family $\mathcal{U}$ is point countable on $D\subset X$ if

\begin{equation*}
|\{U\in\mathcal{U}:\,x\in U\}|\leq \omega_0
\end{equation*}

for every $x\in D$. Since we are interested in compact trees, we are going to state \cite[Proposition 1.9]{Kalenda2} in these terms.

\begin{thm},\label{T0SeparatingFamily}
Let $T$ be a tree and $D=\{t\in T:\cf(t)\leq \omega_0\}$. Then the following are equivalent:
\begin{enumerate}[$(i)$]
\item $T$ is Valdivia.
\item $D$ is a $\Sigma$-subset of $T$.
\item there is a $T_0$-separating family of basic clopen sets point-countable exactly on $D$.
\end{enumerate}
\end{thm}

The equivalence between $(i)$ and $(ii)$ follows from \cite[Proposition 3.2]{SomaTree}, while the equivalence between $(ii)$ and $(iii)$ follows by \cite[Theorem 19.11]{KKLP}, observing that a tree $T$ endowed with the coarse wedge topology is a zero-dimensional space.\\
We conclude this section providing a description of the Radon measures on trees, these results are useful to investigate the spaces of continuous functions on trees. In order to do that we observe that a chain complete and rooted tree $T$ endowed with the coarse wedge topology is the Stone space of the Boolean algebra of clopen subsets of $T$, denoted by $\Clop(T)$. We observe that, for a tree $T$ endowed with the coarse wedge topology, $\Clop(T)$ is generated by the family $\{V_t: t \mbox{ is on a successor level}\}$.\\
Combining the previous observation with \cite[Lemma 3.2]{DzaPle}, we are able to prove that every Radon measure on a tree $T$ has metrizable support. We recall that a partial order is $\sigma$-centred if it is a countable union of centered subsets. 

\begin{lem}\cite[Lemma 3.2]{DzaPle}
Suppose that a Boolean algebra $\U$ is generated by a subfamily $\G$ such that if $a,b\in \G$, then $a\leq b$, $b\leq a$ or $a\cdot b=0$, and that $\U$ is not $\sigma$-centred. Then $\U$ carries no strictly positive measure.
\end{lem} 

\begin{prop}\label{Radonseparablesupport}
Let $T$ be a tree and $\mu$ be a Radon measure on $T$. Then the support of $\mu$ is metrizable.
\end{prop}

\begin{proof}
Let $S=\supp(\mu)$. Hence $S$ is a compact subspace of $T$, therefore it is a Stone space. We observe that the Boolean algebra $\Clop(S)$ is generated by the family $\G=\{V_t\cap S:  t \mbox{ is on a successor level}\}$. Hence by the previous lemma we have that $\Clop(S)$ and in particular $\G$ are $\sigma$-centred.\\
Hence $\G$ may be written as a union of countably many chains. We claim that all these chains are countable. If not let $\{V_{t_{\alpha}}\cap S:\alpha<\omega_1\}$ be a chain of size $\omega_1$,  such that $t_{\alpha}\leq t_{\beta}$ if $\alpha\leq \beta$. Let $U_{\alpha}=(V_{t_{\alpha+1}}\setminus V_{t_{\alpha}})\cap S$ for each $\alpha <\omega_1$. Then $\{U_{\alpha}\}_{\alpha<\omega_1}$ is an uncountable family of disjoint open subsets of $S$ with positive measure. That is a contradiction. Hence $\G$ is countable as well as $\Clop(S)$. Therefore $S$ is metrizable.
\end{proof}

As immediate consequence of the previous result we obtain, by \cite[Theorem 5.3]{Kalenda2}, $C(T)$ is a WLD Banach space if and only if $T$ is a Corson compact space. Moreover, from the previous proposition, we easily obtain the following result.

\begin{cor}\label{supportomisure}
Let $T$ be a tree with height equal to $\eta+1$, where $\cf(\eta)\geq \omega_1$, and $\mu$ be a continuous Radon measure on $T$. Then there exists $\beta<\eta$ such that $\supp(\mu)\subset T_{\beta}$.
\end{cor}

\section{Characterization of Valdivia compact trees}

The purpose of this section is to describe the relations between trees and the class of Valdivia compacta. We will characterize trees of height less than $\omega_2$. We recall the definition of $\omega_1$-relatively discrete subset.

\begin{defn}
Let $X$ be a topological space. We say that a subset $A\subset X$ is \textit{$\omega_1$-relatively discrete} if it can be written as union of $\omega_1$-many relatively discrete subsets of $X$.
\end{defn}

The main results of this section are contained in the following theorem.

\begin{thm}\label{caratValdiviaGeneral}
Let $T$ be a tree. Let $R=\{t\in T:\cf(t)=\omega_1\, \&\, \ims(t)\neq \emptyset\}$. Consider the following conditions:
\begin{enumerate}[$(i)$]
\item $|\ims(t)|<\omega_0$ for every $t\in R$;
\item $R\cap \Lev_{\alpha}(T)$ is $\omega_1$-relatively discrete for each $\alpha<\omega_2$ with $\cf(\alpha)=\omega_1$.
\end{enumerate}
Then the following two statements hold.
\begin{enumerate}[$(1)$]
\item If $T$ is Valdivia, then $\htte(T)\leq \omega_2$ and $(i),(ii)$ hold.
\item If $\htte(T)<\omega_2$ and $(i),(ii)$ hold, then $T$ is Valdivia.
\end{enumerate}
\end{thm}

The proof of this theorem is split in two parts. Since the first statement does not require any extra result we prove it here, while we postpone the second part at the end of the section because two lemmata are needed. 

\begin{proof}[Proof of Theorem \ref{caratValdiviaGeneral}, $(1)$]
Let $T$ be a Valdivia compact tree. Since $T$ is Valdivia, by \cite[Theorem 4.1]{SomaTree}, we have $\htte(T)\leq \omega_2$. Moreover, $T$ has a retractional skeleton, hence by \cite[Theorem 3.1]{SomaTree}, we have $|\ims(t)|<\omega_0$ for every $t\in \Lev_{\alpha}(T)$ with $\cf(\alpha)=\omega_1$. Hence, in particular $|\ims(t)|<\omega_0$  for every $t\in R$ and $(i)$ is fulfilled. We shall prove that $(ii)$ holds as well. In order to do that, let $\alpha <\omega_2$ and $\cf(\alpha)=\omega_1$. Since $T$ is a Valdivia compact space, by Theorem \ref{T0SeparatingFamily} the subtree $T_{\alpha+1}$ is Valdivia as well. Hence, by point $(iii)$ of the same theorem, there exists a family $\mathcal{U}_{\alpha}$ of clopen subsets of $T_{\alpha+1}$ that is $T_0$-separating and point countable on $D_{\alpha}=\{t\in T_{\alpha+1}: \cf(t)\leq \omega_0\}$. Each element $U\in\mathcal{U}_{\alpha}$ is of the form $W_{s}^{F}$ for some $s\in T_{\alpha+1}$ and a finite subset $F\subset T_{\alpha+1}$, whose elements are bigger than $s$ and on a successor level.\\
For every $t\in R\cap\Lev_{\alpha}(T)$ there is $\eta(t)<\alpha$, such that if $U\in\mathcal{U}_{\alpha}$, $t\in U$, $\htte(\min U,T_{\alpha+1})>\eta(t)$, then $U\cap \ims(t)=\emptyset$. Indeed, since for every $s\in \ims(t)$, $s\in D_{\alpha}$, we have that $s$ is contained in countably many elements of $\mathcal{U}_{\alpha}$. For this reason there are only countably many elements of $\mathcal{U}_{\alpha}$ containing both $t$ and $s$. It is enough to take 
\begin{equation*}
\eta(t)=\sup\{\htte(p,T_{\alpha+1}): p<t, \,	\exists\,\, W_{p}^{F}\in\mathcal{U_{\alpha}}: t\in W_{p}^{F},\ims(t)\cap W_{p}^{F}\neq \emptyset\}.
\end{equation*}
Let $R_{\eta}=\{t\in R\cap\Lev_{\alpha}(T):\eta(t)=\eta\}$.\\
Let $t\in R_{\eta}$. Since $t\notin D_{\alpha}$, there exists an unbounded subset $S_t$ of $\hat{t}$ such that for each $s\in S_t$ there exists $W_{s}^{F}\in \mathcal{U}_{\alpha}$ with $t\in W_{s}^{F}$. In particular, since $S_t$ is unbounded, there exists $s_0\in S_t$, and an open basic subset $W_{s_0}^{F}\in \mathcal{U}_{\alpha}$,  with $\htte(s_0,T_{\alpha+1})>\eta$. Since $F$ is finite and $\ims(p)\cap W_{s_0}^{F}=\emptyset$ if $p\in R_\eta$, we have that $|W_{s_0}^{F}\cap R_{\eta}|<\omega_0$. Therefore there exists $r\in T_{\alpha+1}$ on a successor level such that $s_0\leq r<t$ and $V_r\cap R_{\eta}=\{t\}$. Hence $R_{\eta}$ is relatively discrete for each $\eta<\omega_1$, which gives us the assertion.
\end{proof}

We observe that the second statement of Theorem \ref{caratValdiviaGeneral} cannot be reversed. Indeed, there are several examples of Valdivia trees with height equal to $\omega_2$. Here we provide an easy example of such a space. Let $X$ be the topological sum of the ordinal intervals $X_{\alpha}=[0,\alpha]$ where $\alpha<\omega_2$. Let $X_{0}=X\cup\{\infty\}$ be the one-point compactification of $X$. By \cite[Theorem 3.35]{Kalenda2}, $X_0$ is a Valdivia compact space. Consider the following relation on $X_0$:
\begin{itemize}
\item $\infty$ is the least element,
\item $x<y$ in $X$ if and only if there exists $\alpha<\omega_2$ such that $x,y\in X_{\alpha}$ and $x<y$ in $X_{\alpha}$.
\end{itemize}
It is clear that $(X_0,<)$ is a tree and, if it is endowed with the coarse wedge topology, is homeomorphic to $X_0$ with the topology given by the compactification. Therefore we obtained the desired tree.\\
Much more interesting is the following problem that as far as we know seems to be open.
\begin{problem}
Can the first statement of Theorem \ref{caratValdiviaGeneral} be reversed?
\end{problem}

In order to prove the second statement of Theorem \ref{caratValdiviaGeneral}, we need to describe a natural way to extend relatively open subsets to the whole tree. Let $T$ be a tree of height equal to $\alpha$ and let $\beta<\alpha$ be on a successor level. Let $U\subset T_{\beta}$ be a relatively open set in $T_{\beta}$. We extend $U$ to the whole tree as follows:
\begin{equation*}
\widetilde{U}=U\cup (\bigcup_{x\in \Lev_{\beta}(T)\cap U}V_x).
\end{equation*}
It is clear that $\widetilde{U}$ is open in $T$. Given a family $\mathcal{U}_{\beta}$ of open subsets of $T_{\beta}$ we denote by $\widetilde{\mathcal{U}}_{\beta}$ the family of the extended elements of $\mathcal{U}_{\beta}$.\\
Given a family $\mathcal{U}$ of clopen subsets of $T$ we put $\mathcal{U}(t)=\{U\in\mathcal{U}:t\in U\}$, for every $t\in T$. If $A,B\subset T$ and $A\cap B=\{t\}$, by an abuse of the notation, $\mathcal{U}(A\cap B)$ means $\mathcal{U}(t)$. We need three technical lemmata.\\
Let $T$ be a tree with height less or equal to $\alpha +1$, where $\alpha$ has uncountable cofinality. Let $\{\alpha_{\gamma}\}_{\gamma<\cf(\alpha)}$ be a continuous increasing transfinite sequence converging to $\alpha$. Let us denote by $I(\cf(\alpha))$ the subset of all successor ordinals less than $\cf(\alpha)$.\\
Suppose that for each $\gamma\in I(\cf(\alpha))$, there exists a $T_{0}$-separating family $\U_{\gamma}$ in $T_{\alpha_{\gamma}+1}$. Moreover, suppose that each element $t\in T_{\alpha_{\gamma}+1}$ belongs to at least one element of the family $\U_{\gamma}$.  For each $\gamma \in I(\cf(\alpha))$, $U\in \U_{\gamma}$ and $t\in \Lev_{\alpha_{(\gamma-1)}+1}(T)$ define $U_{t}=V_{t}\cap U$. Finally, we define a family $\U$ as follows:
\begin{equation*}
\U=\bigcup_{\gamma\in I(\cf(\alpha))}\bigcup_{U\in\mathcal{U}_{\gamma}}\{\widetilde{U}_t: t\in \Lev_{\alpha_{(\gamma-1)}+1}(T)\}.
\end{equation*} 
Now we can state the first lemma.

\begin{lem}\label{T0Ufamily}
Let $T$ be a tree with height less or equal to $\alpha + 1$, where $\alpha$ has uncountable cofinality. Let $\U$ be the family of clopen subsets of $T$ defined as above. Then $\U$ is $T_{0}$-separating in $T$.
\end{lem}

\begin{proof}
Let $s,t\in T$ such that $\htte(s,T)\leq \htte(t,T)$:
\begin{itemize}
\item if $t\in T_{\alpha_{\gamma}+1}\setminus T_{\alpha_{(\gamma-1)}}$ for some $\gamma \in I(\cf(\alpha))$, then the assertion follows from the fact that the family $\U_{\gamma}$ is $T_0$-separating in  $T_{\alpha_{\gamma}+1}$;
\item if $t\in \Lev_{\alpha_{\gamma}}(T)$ with $\gamma$ limit, then we observe that $\htte(s\wedge t,T)<\htte(t,T)$. Since $\alpha_{\gamma}$ is limit too, there is $\xi<\gamma$ successor ordinal such that $\htte(s\wedge t,T)<\alpha_{\xi-1}$. We define $\{u\}=\hat{t}\cap \Lev_{\alpha_{\xi}+1}(T)$, let us consider two cases. If $\{v\}=\hat{s}\cap\Lev_{\alpha_{\xi}+1}(T)$, then, since $u,v\in T_{\alpha_{\xi}+1}$ and $\U_{\xi}$ is $T_{0}$-separating in $T_{\alpha_{\xi}+1}$, there is $U\in \U_{\xi}$ such that $|U\cap\{u,v\}|=1$. It follows that $|\widetilde{U}\cap \{s,t\}|=1$. Otherwise suppose that $\hat{s}\cap\Lev_{\alpha_{\xi}+1}(T)=\emptyset$. Similarly there exists $U\in \U_{\xi}$ such that $|U\cap\{u,s\}|=1$. Thus we have $|\widetilde{U}\cap \{s,t\}|=1$;
\item if $t\in \Lev_{\alpha}(T)$ we use the same argument as in the previous item.
\end{itemize}
Therefore $\U$ is a $T_{0}$-separating family in $T$.
\end{proof}

\begin{lem}\label{LemNum}
Let $T$ be a tree with height greater than $\eta$, where $\cf(\eta)\geq \omega_1$. Let $N$ be a countable subset of $\Lev_{\eta}(T),$ then there exists $\delta<\eta$ such that if $t_1,t_2\in N$, then $\htte(t_1\wedge t_2,T)<\delta$.
\end{lem}

\begin{proof}
Let $N=\{t_n\}_{n\in \omega_0}$ and suppose that $\{t_n\}_{n\in \omega_0}$ is a one-to-one sequence. Define $\delta_{n}^{m}=\htte(t_n\wedge t_m,T)$. The assertion follows taking $\delta=(\sup_{n,m\in\omega_0,n\neq m}\delta_{n}^{m})+1$. If $N$ is finite, we use the same argument as in the infinite case. 
\end{proof}

\begin{lem}\label{ValdiviaEta}
Let $T$ be a tree of height less or equal to $\eta +2$ where $\eta <\omega_2$ and $\cf(\eta)=\omega_1$. Suppose that:
\begin{enumerate}
\item $R=\{t\in \Lev_{\eta}(T): \ims(t)\neq\emptyset\}$ has cardinality at most $\omega_1$;
\item $|\ims(t)|<\omega_0$ for every $t\in R$;
\item $T_{\gamma+1}$ is a Valdivia compactum for every $\gamma<\eta$ .
\end{enumerate}
Then $T$ is a Valdivia compact space.
\end{lem}

\begin{proof}
We split the proof in two parts. In the first part we define a function $\theta:[0,\omega_1)\to[0,\eta)$ satisfying certain properties. This function $\theta$ will be defined separately in three different cases: when $R$ is uncountable, infinite countable an finite. In the second part we use the mapping $\theta$ in order to define a suitable family $\U$ of clopen subsets of $T$ that is $T_0$-separating and point-countable on $D=\{t\in T: \cf(t)\leq \omega_0\}$.\\
Suppose that $|R|=\omega_1$, then we enumerate it as $R=\{t_{\alpha}\}_{\alpha<\omega_1}$. Let $\{\eta_{\gamma}\}_{\gamma<\omega_1}$ be a continuous increasing transfinite sequence converging to $\eta$. We may suppose that $\eta_0=0$. Let us define the mapping $\theta$ by using a transfinite recursion argument. Let $\theta(0)=0$ and for each $\zeta<\omega_1$ we set
\begin{equation*}
\theta(\zeta)=\max(\eta_{\zeta},\sup\{\htte(t_{\beta}\wedge t_{\gamma},T)+1:\beta,\gamma<\zeta, t_{\beta}\neq t_{\gamma}\},\sup\{\theta(\xi)+1:\xi<\zeta\}).
\end{equation*}
We observe that $\theta$ satisfies the following conditions:
\begin{itemize}
\item for every $\alpha<\omega_1$ and $\beta,\gamma< \alpha$ $(t_{\beta}\neq t_{\gamma})$, $\htte(t_{\beta}\wedge t_{\gamma},T)<\theta (\alpha)$,
\item $\theta$ is increasing, continuous and $\sup_{\zeta<\omega_1}\theta(\zeta)=\eta$.
\end{itemize}
Let us prove that $\theta$ is continuous, the other properties of $\theta$ are clear. Let $\zeta<\omega_1$ be a limit ordinal, we need to show that $\sup_{\xi<\zeta}\theta(\xi)=\theta(\zeta)$. We observe that $\eta_{\zeta}=\sup_{\xi<\zeta}\eta_{\xi}\leq \sup_{\xi<\zeta}(\theta(\xi)+1)$, furthermore we have $\sup\{\htte(t_{\beta}\wedge t_{\gamma},T)+1: \beta, \gamma<\zeta, t_{\beta}\neq t_{\gamma}\}=\sup_{\xi<\zeta}\sup\{\htte(t_{\beta}\wedge t_{\gamma},T)+1: \beta, \gamma<\xi, t_{\beta}\neq t_{\gamma}\}\leq \sup_{\xi<\zeta}\theta(\xi)\leq \sup_{\xi<\zeta}(\theta(\xi)+1)$. Hence, by definition of $\theta(\zeta)$, we obtain $\sup_{\xi<\zeta}(\theta(\xi)+1)=\theta(\zeta)$. This proves the continuity.\\
Suppose that $|R|=\omega_0$, then we enumerate it as $R=\{t_{n}\}_{n<\omega_0}$. Let $\{\theta(\alpha)\}_{\alpha<\omega_1}$ be any continuous increasing transfinite sequence that satisfies the two following conditions $\theta(0)=0$ and $\theta(n)=\sup\{\htte(t_{m_1}\wedge t_{m_2},T):m_1<m_2<n\}$ for each $n\leq \omega_0$. Similarly it is possible to define the function $\theta$ when $R$ is finite.\\
We observe that for every $t\in T$ such that $\htte(t,T)<\eta$ and $t$ on a successor level, there exists a unique $\alpha<\omega_1$ such that $\htte(t,T)\in [\theta(\alpha),\theta(\alpha+1))$. Moreover, under the same hypothesis, there exists at most one $\beta< \alpha$ such that $t<t_{\beta}$.\\
Since $T_{\theta(\alpha)+1}$ is Valdivia, by Theorem \ref{T0SeparatingFamily}, there exists a family $\mathcal{U}_{\alpha}$ of clopen subsets of $T_{\theta(\alpha)+1}$ which is $T_0$-separating and point-countable on $D_{\alpha}=\{t\in T_{\theta(\alpha) +1}: \cf(t)\leq \omega_0\}$ for every $\alpha<\omega_1$. Moreover, the elements of $\mathcal{U}_{\alpha}$ are of the form $W_{s}^{F}$ for every $\alpha<\omega_1$. Finally we may suppose that each element of $T_{\theta(\alpha)+1}$ is contained in some element of the family $\mathcal{U}_{\alpha}$ (for example adding to the family $\mathcal{U}_{\alpha}$ the element $T_{\theta(\alpha)+1}$).\\
In order to define a family $\mathcal{U}$ of clopen subsets of $T$ which is $T_{0}$-separating and point-countable on $D$, we are going to select and opportunely modify a suitable subfamily of $\bigcup_{\alpha<\omega_1}\mathcal{U}_{\alpha}$.\\
Let $\alpha<\omega_1$ be a successor ordinal and $U\in \mathcal{U}_{\alpha}$. For every $t\in \Lev_{\theta(\alpha - 1)+1}(T)$ let $U_t=U\cap V_t$. We recall that if $t\in \Lev_{\theta(\alpha-1)+1}(T)$, then $|V_t\cap \{t_{\beta}\}_{\beta<\alpha-1}|\leq 1$. Therefore if $V_{t}\cap \{t_{\beta}\}_{\beta<\alpha-1}=\emptyset$, then we extend $U_t$ as $\widetilde{U}_t$, while if 
$V_{t}\cap\{t_{\beta}\}_{\beta<\alpha-1}=\{t_{\gamma}\}$ for some $\gamma<\alpha-1$, then we extend $U_t$ as $\widetilde{U}_t\setminus \ims(t_{\gamma})$ obtaining a clopen subset of $T$ that avoid the subset $\bigcup_{\beta<\alpha-1}\ims(t_{\beta})$.\\
Define $I(\omega_1)$ as the set of successor ordinals less than $\omega_1$. We define the following family of clopen subsets of $T$:
\begin{equation*}
\begin{split}
\mathcal{U}=&\{\{t\}: t\in \Lev_{\eta+1}(T)\}\\
&\cup\bigcup_{\alpha\in I(\omega_1)}\bigcup_{U\in\mathcal{U}_{\alpha}}\{\widetilde{U}_t: t\in \Lev_{\theta(\alpha-1)+1}(T),\,V_t\cap \{t_{\beta}\}_{\beta< \alpha-1 }=\emptyset\}\\
&\cup\bigcup_{\alpha\in I(\omega_1)}\bigcup_{U\in\mathcal{U}_{\alpha}}\{\widetilde{U}_t\setminus\ims(t_{\xi}): t\in \Lev_{\theta(\alpha-1)+1}(T),\, \{t_{\xi}\}= V_t\cap \{t_{\beta}\}_{\beta< \alpha-1 }\}.
\end{split}
\end{equation*}
Now we are going to prove that $T$ is Valdivia. We observe that the family $\mathcal{U}$ restricted to $T_{\eta}$ satisfies the hypothesis of Lemma \ref{T0Ufamily}. Therefore combining Lemma \ref{T0Ufamily} with the fact that the family $\{\{t\}: t\in \Lev_{\eta+1}(T)\}$ is contained in $\U$ we obtain that $\U$ is $T_0$-separating in $T$.\\
It remains to prove that $\mathcal{U}$ is point-countable on $D$. Suppose that $t\in D$, then we consider the following two cases:
\begin{itemize}
\item suppose that $\htte(t,T)<\eta$. We observe that $|\{\alpha<\omega_1:\htte(t,T)>\theta(\alpha)\}|\leq\omega_0$, let us define $\alpha_0=\sup\{\alpha<\omega_1:\htte(t,T)>\theta(\alpha)\}$. Hence if $t\in U$ and $U\in \mathcal{U}$ we have that $U$ is extended from an element of a family $\mathcal{U}_{\xi}$ where $\xi\leq \alpha_0$. Since $\mathcal{U}_{\xi}$ is point-countable on $D_{\xi}\subset T_{\theta(\xi)+1}$ we have:
\begin{equation*}
|\mathcal{U}(t)|\leq|\bigcup_{\xi\leq\alpha_0}\mathcal{U}_{\xi}((\hat{t}\cap\Lev_{\theta(\xi)+1}(T)))|\leq\omega_0.
\end{equation*}
\item Suppose that $\htte(t,T)=\eta+1$, hence there exists $t_{\beta}\in R$ such that $t\in\ims(t_{\beta})$ for some $\beta<\omega_1$. Let $X\in\mathcal{U}$ such that $t\in X$. Then there are the following possibilities:
\begin{enumerate}
\item $X=\{t\}$, exactly one element of $\mathcal{U}$ has this form;
\item there exist $\xi\in I(\omega_1)$ and $s\in\Lev_{\theta(\xi-1)+1}(T)$ such that $X=\widetilde{U}_s$, for some $U\in \mathcal{U}_{\xi}$. Since $V_s\cap \{t_{\gamma}\}_{\gamma<\xi-1}=\emptyset$ we obtain $\xi\leq\beta+1$, moreover we observe that: $\hat{t}\cap\Lev_{\theta(\xi)+1}(T)\subset U_s$ and $|\mathcal{U}_{\xi}(\hat{t}\cap \Lev_{\theta(\xi)+1}(T))|\leq \omega_0$. Hence there are at most countably many elements of this form;
\item there are $\xi\in I(\omega_1)$, $s\in \Lev_{\theta(\xi-1)+1}(T)$ and $p\in R$ such that $X=\widetilde{U}_s\setminus\ims(p)$. Since $V_{s}\cap\bigcup_{\gamma<\xi-1}\ims(t_{\gamma})=\emptyset$, we have $\xi\leq\beta+1$ and since $(\hat{t}\cap\Lev_{\theta(\xi)+1}(T))\subset U_s$ and $|\mathcal{U}_{\xi}(\hat{t}\cap \Lev_{\theta(\xi)+1}(T))|\leq \omega_0$, there are at most countably many sets of this form.
\end{enumerate}
\end{itemize}
Therefore $\mathcal{U}$ is point-countable on $D$, hence $T$ is Valdivia.
\end{proof}

Now we are ready to prove the statement $(2)$ in Theorem \ref{caratValdiviaGeneral}.

\begin{proof}[Proof of Theorem \ref{caratValdiviaGeneral}, $(2)$]
In order to prove the second part of the theorem we are going to use a transfinite induction argument on the height of the tree. Let $T$ be a tree as in the hypothesis, by \cite[Theorem 4.1]{SomaTree}, if $\htte(T)\leq\omega_1+1$, then $T$ is Valdivia.\\
\\
Suppose that the assertion is true for each tree $T$ that satisfies $\htte(T)\leq \alpha +2$. Then we will prove the assertion for each tree $T$ that satisfies $\htte(T)\leq \alpha+3$.\\
Let $T$ be a tree that satisfies $\htte(T)= \alpha+3$, then, by induction hypothesis, $T_{\alpha +1}$ is a Valdivia compact space. Hence, by Theorem \ref{T0SeparatingFamily}, there exists a family $\mathcal{U}_{\alpha}$ of clopen subsets of $T_{\alpha}$ which is $T_0$-separating and point-countable on $D_{\alpha}=\{t\in T_{\alpha +1}: \cf(t)\leq \omega_0\}$. The family $\mathcal{U}=\widetilde{\mathcal{U}}_{\alpha}\cup\{\{t\}:t\in  \Lev_{\alpha +2}(T)\}$ is a family of clopen subset of $T$. It is easy to prove that $\mathcal{U}$ is $T_0$-separating and point-countable on $D=\{t\in T: \cf(t)\leq \omega_0\}$. Therefore $T$ is Valdivia.\\
\\
Suppose that the assertion is true for each tree $T$ that satisfies $\htte(T)< \alpha$ for some limit ordinal $\alpha$, then we will prove the assertion for each tree $T$ that satisfies $\htte(T)\leq \alpha +2$.  Therefore, suppose $T$ is a tree of height less or equal to $\alpha+2$. Let us consider the two different cases.\\ 
Suppose that $\alpha$ is a limit ordinal with countable cofinality. Then there exists an increasing sequence of ordinals $\{\alpha_n\}_{n\in\omega_0}$ converging to $\alpha$. By induction hypothesis the subtrees $T_{\alpha_{n}+1}$ are Valdivia compact spaces. Hence, by Theorem \ref{T0SeparatingFamily}, there exists a $T_0$-separating family of clopen subsets $\mathcal{U}_{n}$ that is point countable on $D_{n}=\{t\in T_{\alpha_n +1}: \cf(t)\leq \omega_0\}$, for each $n\in\omega_0$. Now we are going to prove that the family $\mathcal{U}=\bigcup_{n\in\omega_0}\widetilde{\mathcal{U}}_{n }\cup \{\{t\}:t\in \Lev_{\alpha+1}(T)\},$ is a $T_0$-separating family of clopen subset of $T$ that is point countable on $D=\{t\in T: \cf(t)\leq \omega_0\}$. Suppose that $t\in T$, then let us consider the three possibilities:
\begin{itemize}
\item if $t\in \Lev_{\alpha}(T)$, then $\mathcal{U}(t)=\bigcup_{n\in\omega_0}\{\widetilde{U}:U\in\mathcal{U}_{n}(\hat{t}\cap\Lev_{\alpha_n+1}(T))\}$;
\item if $t\in \Lev_{\alpha+1}(T)$, then $\mathcal{U}(t)=\{\{t\}\}\cup\bigcup_{n\in\omega_0}\{\widetilde{U}:U\in\mathcal{U}_{n}(\hat{t}\cap\Lev_{\alpha_n+1}(T))\}$;
\item if $t\in D\cap \bigcup_{n\in\omega_0} T_{\alpha_n+1}$, then $\mathcal{U}(t)=\bigcup_{n\in\omega_0}\{\widetilde{U}:U\in\mathcal{U}_n(t)\}$;
\end{itemize}
in all cases $\mathcal{U}(t)$ is countable, hence $\mathcal{U}$ is point-countable on $D$. Let us prove that $\U$ is $T_0$-separating on $T$. For this purpose let $s,t\in T$ be two elements that satisfy $\htte(s,T)\leq \htte(t,T)$:
\begin{itemize}
\item if $\htte(t,T)<\alpha$, then $s,t \in T_{\alpha_n+1}$ for some $n\in \omega_0$. The assertion follows from the fact that the family $\U_n$ is $T_0$-separating on $T_{\alpha_n+1}$;
\item if $t\in\Lev_{\alpha}(T)$, then $\htte(s\wedge t, T)<\alpha$, hence $s\wedge t\in T_{\alpha_n+1}$, for some $n\in\omega_0$. Let $\{t_1\}=\hat{t}\cap \Lev_{\alpha_{n+1}+1}(T)$ and $\{s_1\}=\hat{s}\cap \Lev_{\alpha_{n+1}+1}(T)$. Then there exists $U\in\U_{n+1}$ such that $|U\cap\{s_1,t_1\}|=1$, then the assertion follows observing that $\widetilde{U}\in \U$ and $|\widetilde{U}\cap\{s,t\}|=1$; 
\item if $t\in \Lev_{\alpha +1}(T)$, then the assertion follows observing that $\{t\}\in \U$. 
\end{itemize}
Suppose that $\alpha$ is a limit ordinal with uncountable cofinality. Then there exists an increasing continuous transfinite sequence of ordinals $\{\alpha_{\gamma}\}_{\gamma<\omega_1}$ converging to $\alpha$. Since $T$ satisfies $(ii)$ we have that $\Lev_{\alpha}(T)\cap R=\bigcup_{\xi<\omega_1}A_{\xi}$, where $A_{\xi}$ is relatively discrete in $T$ for each $\xi<\omega_1$, we may suppose that the family $\{A_{\xi}\}_{\xi<\omega_1}$ is disjoint. We observe that any relatively discrete subset $B\subset \Lev_{\alpha}(T)$ can be decomposed as $B=\bigcup_{\beta<\omega_1}B_{\beta}$ in such a way that if $s\in T$ and $\htte(s,T)>\alpha_{\beta+1}$, then $V_{s}\cap B_{\beta}$ contains at most one point. Indeed, since $B$ is relatively discrete, for each $t\in B$ there exists $s_t<t$ on a successor level such that $V_{s_t}\cap B=\{t\}$. Define $B_{\beta}=\{t\in B: \alpha_{\beta}<\htte(s_{t},T)\leq\alpha_{\beta+1}\}$. Hence we have $B=\bigcup_{\beta<\omega_1}B_{\beta}$ and if $s\in T$ is such that $\htte(s,T)>\alpha_{\beta+1}$, then we have $|V_s\cap B_{\beta}|\leq 1$. Therefore, since each $A_{\xi}$, where $\xi<\omega_1$, is a relatively discrete subset of $\Lev_{\alpha}(T)$, such a $A_{\xi}$ can be decomposed in $\omega_1$-many pieces as above. Hence we may suppose, without loss of generality, that for each $\xi<\omega_1$ there is $\beta(\xi)<\alpha$ such that for any $s\in T$ with $\htte(s,T)>\beta(\xi)$ we have $|V_s\cap A_{\xi}|\leq 1$.
Moreover we may suppose that the function $\beta$ is non-decreasing (replace $\beta(\xi)$ by $\sup\{\beta(\gamma):\gamma\leq\xi\}$).\\
Firstly let us suppose that the function $\beta$ is bounded by an ordinal $\beta_{0}<\alpha$.\\
Let $p\in \Lev_{\beta_0 +1}(T)$. Since the height of $p$ is greater than $\beta_0$, we have $|V_{p}\cap A_{\xi}|\leq 1$ for every $\xi<\omega_1$. Whence we get $|V_p\cap\bigcup_{\xi<\omega_1} A_{\xi}|\leq \omega_1$. By induction hypothesis $T_{\beta_0+1}$ is Valdivia, hence there exists a family $\mathcal{U}_{0}$ of clopen subsets of $T_{\beta_0+1}$ which is $T_{0}$-separating and point-countable on $D_{0}=\{t\in T_{\beta_0 +1}:\cf(t)\leq\omega_0\}$. Further, for any $p\in \Lev_{\beta_0 +1}(T)$, the subset $V_p\subset T$ is isomorphic to a tree satisfying the assumptions of Lemma \ref{ValdiviaEta}. Hence $V_p$ is a Valdivia compact space. Therefore there is a family of clopen sets $\mathcal{U}_p$ that is $T_0$-separating and point countable on $D_{p}=\{t\in V_p: \,\cf(t)\leq \omega_0\}$. We may assume without loss of generality that $V_{p}\in \mathcal{U}_{p}$, for every $p\in \Lev_{\beta_0 +1}(T)$. Defining $\mathcal{U}=\widetilde{\mathcal{U}}_0\cup (\bigcup_{p\in \Lev_{\beta_0 +1}(T)}\mathcal{U}_p)$ we obtain a family of clopen subsets of $T$. Let $t\in D=\{t\in T: \cf(t)\leq \omega_0\}$. We consider the two cases:
\begin{itemize}
\item if $\htte(t,T)\geq \beta_0 +1$, take $\{p_t\}=\hat{t}\cap \Lev_{\beta_0+1}(T)$. Then we have
\begin{equation*}
|\mathcal{U}(t)|\leq |\mathcal{U}_0(p_t)| +|\mathcal{U}_{p_t}(t)|\leq \omega_0;
\end{equation*}
\item if $\htte(t,T)< \beta_0 +1$, we have
\begin{equation*}
|\mathcal{U}(t)|\leq |\mathcal{U}_0(t)|\leq \omega_0.
\end{equation*}
\end{itemize}
Hence $\mathcal{U}$ is point-countable on $D$. We continue by proving that $\mathcal{U}$ is $T_{0}$-separating, let $s,t\in T$ and suppose that $\htte(s,T)\leq \htte(t,T)$.
\begin{itemize}
\item suppose that either $s,t\in T_{\beta_0 +1}$ or $s,t\in V_p$ for some $p\in \Lev_{\beta_0+1}(T).$ Then we use the fact that $\mathcal{U}_0$ $(\mathcal{U}_p)$ is $T_0$-separating on $T_{\beta_0+1}$  ($V_p$, respectively);
\item otherwise, suppose that there exists $p\in \Lev_{\beta_0+1}(T)$ such that $t\in V_p$ and $s\notin V_p$. Then we have $V_p\in \mathcal{U}$ and $s\notin V_p$.
\end{itemize}
Hence the family $\mathcal{U}$ is $T_0$-separating. Therefore $T$ is Valdivia.\\
Let us suppose that the mapping $\beta$ is unbounded. We recall that $\beta$ has the following property: if $t\in T$ and $\htte(t,T)>\beta(\xi)$, for some $\xi<\omega_1$, we have $|V_{t}\cap(\cup_{\eta\leq\xi}A_{\eta})|\leq\omega_0$.\\
We are going to define a family $\{S_{\xi}\}_{\xi<\omega_1}$ of subsets of $T$ that satisfies the following properties:

\begin{enumerate}[$(a)$]
\item $S_{\xi}\subset T_{<\alpha}$;
\item $S_{\xi}\cap S_{\eta}=\emptyset$ for $\xi\neq \eta$;
\item if $t\in \bigcup_{\gamma<\xi}S_{\gamma}$ for some $\xi<\omega_1$, then $\{s\in T:s<t\}\subset \bigcup_{\gamma<\xi}S_{\gamma}$;
\item if $t\in T_{<\alpha}\setminus (\bigcup_{\gamma\leq\xi} S_{\gamma}),$ then $\htte(t,T)>\beta(\xi+1)$;
\item if $t\in S_{\xi}$ for some $\xi<\omega_1$, then $V_t\cap(\bigcup_{\gamma<\xi}A_{\gamma})$ is at most a singleton.
\item if $t\in T_{<\alpha}\setminus (\bigcup_{\gamma\leq\xi} S_{\gamma}),$ then $V_t\cap(\bigcup_{\gamma\leq\xi}A_{\gamma})$ is at most a singleton.
\end{enumerate}

We use a transfinite induction argument. Firstly we define
\begin{equation*}
S_0=\{t\in T: \htte(t,T)\leq \beta(1)\}.
\end{equation*} 
We observe that $S_0=T_{\beta(1)}$ and $\beta(1)<\alpha$, hence $S_0$ satisfies $(a)-(f)$. Let us suppose that for every $\gamma<\eta$, $S_{\gamma}$ has been already defined such that $(a)-(f)$ are fulfilled. Now we are going to define $S_{\eta}$. Let us consider the two cases:
\begin{itemize}
\item $\eta=\gamma +1$. Let $M_{\eta}=\{t\in T: t \mbox{ is minimal in } T_{<\alpha}\setminus (\bigcup_{\zeta\leq \gamma}S_{\zeta})\}$. Fix $t\in M_{\eta}$, then we have, by induction hypothesis, $\htte(t,T)>\beta(\eta)$. Hence $|V_t\cap (\cup_{\zeta\leq\eta}A_{\zeta})|\leq \omega_0$, therefore by Lemma \ref{LemNum} there exists $\delta(t)<\alpha$ such that $\delta(t)>\beta(\eta)+1$ and if $t_0,t_1\in V_t\cap (\cup_{\zeta\leq \eta}A_{\zeta})$ we have $\htte(t_0\wedge t_1,T)<\delta(t)$. Let $z(t)=\max\{\delta(t), \beta(\eta+1)\}$ and 
\begin{equation*}
S_{\eta}=\bigcup_{t\in M_{\eta}}(V_t\cap T_{z(t)}).
\end{equation*}
Now we are going to prove that $S_{\eta}$ satisfies $(a)-(f)$. For every $t\in M_{\eta}$ we have $\beta(\eta+1)\leq z(t)<\alpha$, hence $(a)$ and $(d) $ are satisfied.   By construction $S_{\eta}\subset  T_{<\alpha}\setminus (\bigcup_{\zeta\leq \gamma}S_{\zeta})$, therefore $(b)$ is satisfied.\\
By definition of $S_{\eta}$ and the induction hypothesis it follows that if $t\in \bigcup_{\gamma<\eta+1}S_{\gamma}$, then $\{s\in T:s<t\}\subset \bigcup_{\gamma<\eta+1}S_{\gamma}$, hence $S_{\eta}$ satisfies $(c)$.\\
Let $t\in S_{\eta}$. Then $t$ belongs to $T_{<\alpha}\setminus \bigcup_{\zeta\leq \gamma}S_{\zeta},$ hence, by induction hypothesis, we have $|V_t \cap (\cup_{\zeta< \eta}A_{\zeta})|=|V_t \cap (\cup_{\zeta\leq \gamma}A_{\zeta})|\leq 1$. Hence $S_{\eta}$ satisfies $(e)$.\\
Finally we prove that $S_{\eta}$ satisfies $(f)$. Suppose that $t\in T_{<\alpha}\setminus \bigcup_{\zeta\leq \eta}S_{\zeta}$ and $\{p\}=\hat{t}\cap M_{\eta}$. Then we have $\htte(t,T)>z(p)\geq \delta(p)>\htte(t_{0}\wedge t_{1},T)$, for each $t_0,t_1 \in V_p\cap (\cup_{\zeta\leq \eta}A_{\zeta})$. Hence it follows that $|V_t \cap (\cup_{\zeta\leq \eta}A_{\zeta})|\leq 1$;
\item $\eta$ is limit. The function $\beta$ is not necessarily continuous, so we define the set $S_{\eta}$ in two steps. Let $M_{\eta}^{0}=\{t\in T: t \mbox{ is minimal in } T_{<\alpha}\setminus (\bigcup_{\gamma< \eta}S_{\gamma})\}$. Suppose that $t\in M_{\eta}^{0}$, then by induction hypothesis, we have $\htte(t,T)>\beta(\gamma)$, for every $\gamma<\eta$. Hence $|V_t\cap (\cup_{\gamma< \eta}A_{\gamma})|\leq \omega_0$, therefore by Lemma \ref{LemNum} there exists $\delta^0(t)<\alpha$ such that $\delta^0(t)>\sup_{\gamma<\eta}(\beta(\gamma))+1$ and if $t_0,t_1\in V_t\cap (\cup_{\zeta< \eta}A_{\zeta})$ we have $\htte(t_0\wedge t_1,T)<\delta(t)$. Let $z^0(t)=\max\{\delta^0(t), \beta(\eta)\}$ and 
\begin{equation*}
S_{\eta}^{0}=\bigcup_{t\in M_{\eta}^{0}}(V_t\cap T_{z^0(t)}).
\end{equation*}
Let $M_{\eta}=\{t\in T: t \mbox{ is minimal in } T_{<\alpha}\setminus (S_{\eta}^{0} \cup (\bigcup_{\gamma< \eta}S_{\gamma}))\}$. Suppose that $t\in M_{\eta}$, then we have $\htte(t,T)>\beta(\eta)$. Hence $|V_t\cap (\cup_{\zeta\leq\eta}A_{\zeta})|\leq \omega_0$, therefore by Lemma \ref{LemNum} there exists $\delta(t)<\alpha$ such that $\delta(t)>\beta(\eta)+1$ and if $t_0,t_1\in V_t\cap (\cup_{\zeta\leq \eta}A_{\zeta})$ we have $\htte(t_0\wedge t_1,T)<\delta(t)$. Let $z(t)=\max\{\delta(t), \beta(\eta +1)\}$ and 
\begin{equation*}
S_{\eta}=S_{\eta}^{0} \cup \bigcup_{t\in M_{\eta}}(V_t\cap T_{z(t)}).
\end{equation*}
Since the definition of $S_{\eta}$ is similar to the one given in the previous case, conditions $(a)-(d)$ are verified analogously.\\
Suppose that $t\in S_{\eta}$, then $t\in T_{<\alpha}\setminus \bigcup_{\zeta<\gamma}S_{\zeta}$ for each $\gamma<\eta$. Hence, by induction hypothesis, we have $|V_t \cap \cup_{\zeta\leq\gamma}A_{\zeta}|\leq 1$, for each $\gamma<\eta$. It follows that $|V_{t}\cap \cup_{\zeta<\eta}A_{\zeta}|\leq 1$. Therefore $S_{\eta}$ satisfies $(e)$.\\
Finally we prove that $S_{\eta}$ satisfies $(f)$. Let $t\in T_{<\alpha}\setminus \bigcup_{\gamma\leq \eta}S_{\gamma}$. If $\{p\}=\hat{t}\cap M_{\eta}$, then $\htte(t,T)>z(p)\geq \delta(p)>\htte(t_0\wedge t_1, T)$ for each $t_{0},t_{1}\in V_p\cap (\cup_{\zeta\leq \eta}A_{\zeta})$. Hence we have $|V_t \cap (\cup_{\zeta\leq \eta}A_{\zeta})|\leq 1$. Arguing in an analogous way the same follows if $\{p\}=\hat{t}\cap M_{\eta}^{0}.$
\end{itemize}
By the transfinite induction hypothesis the tree $T_{\alpha_{\gamma}+1}$ is a Valdivia compact space, hence, by Theorem \ref{T0SeparatingFamily}, there is a family $\mathcal{U}_{\gamma}$ of clopen subsets of $T_{\alpha_{\gamma}+1}$ which is $T_{0}$-separating and point-countable on $D_{\gamma}=\{t\in T_{\alpha_{\gamma}+1}:\cf(t)\leq\omega_0\}$, for every $\gamma<\omega_1$. Moreover we have that the elements of each $\mathcal{U}_{\gamma}$ are of the form $W_{s}^{F}$.\\
By construction $\{S_{\xi}\}_{\xi<\omega_1}$ is a pairwise disjoint family of subsets of $T_{<\alpha}$. Indeed, since $\beta $ is unbounded, for every $t\in T_{<\alpha}$ there is $\xi<\omega_1$ with $\htte(t,T)<\beta(\xi)$ and so by $(d)$, $t\in \bigcup_{\gamma<\xi}S_{\xi}$. Thus, taking into account $(b)$, for every $t\in T_{<\alpha}$ there exists a unique $\xi<\omega_1$ with $t\in S_{\xi}$.\\
Let $\phi:T_{<\alpha}\to [0,\omega_1)$ be the function that satisfies $t\in S_{\phi(t)}$, for every $t\in T_{<\alpha}$. Let $I(\omega_1)$ be the set of successors ordinals less than $\omega_1$.\\
Let $\gamma \in I(\omega_1)$. Define $U_p=V_p\cap U$ for any  $U\in\mathcal{U}_{\gamma}$ and $p\in \Lev_{\alpha_{(\gamma-1)}+1}(T)$. We observe that $\min(U_p)\in S_{\phi(\min(U_p))}$, hence we obtain from $(e)$ that $|\widetilde{U}_p \cap \cup_{\gamma<\phi(\min(U_p))}A_{\gamma}|\leq 1$. If $\widetilde{U}_p\cap \cup_{\gamma<\phi(\min(U_p))}A_{\gamma}=\emptyset$, then $\widetilde{U}_p$ is a clopen subset of $T$ that does not intersect the set $\cup_{\gamma<\phi(\min(U_p))} A_{\gamma}$. Similarly, if $\widetilde{U}_p\cap \cup_{\gamma<\phi(\min(U_p))}A_{\gamma}=\{s\}$, then $\widetilde{U}_p\setminus \ims(s)$ is a clopen subset of $T$ that does not intersect the set $\cup_{\gamma<\phi(\min(U_p))}\cup_{s\in A_{\gamma}}\ims(s)$. Therefore we define a family $\mathcal{U}$ of clopen subsets of $T$ as follows:
\begin{equation*}
\begin{split}
\mathcal{U}=&\{\{t\}: t\in \Lev_{\alpha+1}(T)\}\\
&\cup\bigcup_{\gamma\in I(\omega_1)}\bigcup_{U\in\mathcal{U}_{\gamma}}\{\widetilde{U}_p: p\in \Lev_{\alpha_{(\gamma-1)}+1}(T), \, \widetilde{U}_p\cap \bigcup_{\eta<\phi(\min(U_p))} A_{\eta}=\emptyset\}\\
&\cup\bigcup_{\gamma\in I(\omega_1)}\bigcup_{U\in\mathcal{U}_{\gamma}}\{\widetilde{U}_p\setminus\ims(s): p\in \Lev_{\alpha_{(\gamma-1)}+1}(T),\, \widetilde{U}_p\cap \bigcup_{\eta<\phi(\min(U_p))} A_{\eta}=\{s\}\}.
\end{split}
\end{equation*}
It remains to prove that $\mathcal{U}$ is a  family of clopen subsets which is $T_{0}$-separating and point-countable on $D=\{t\in T:\cf(t)\leq\omega_0\}$. We observe that the family $\mathcal{U}$ restricted to $T_{\alpha}$ satisfies the hypothesis of Lemma \ref{T0Ufamily}. Therefore combining Lemma \ref{T0Ufamily} with the fact that the family $\{\{t\}: t\in \Lev_{\eta+1}(T)\}$ is contained in $\U$ we obtain that $\U$ is $T_0$-separating in $T$.\\
Now we are going to prove that $\mathcal{U}$ is point-countable on $D$. Let $t\in D$ and consider the following two cases.
\begin{itemize}
\item Suppose that $\htte(t,T)<\alpha$. We define $\gamma_0=\min\{\gamma<\omega_1:\htte(t,T)<\alpha_{\gamma}\}$, such a $\gamma_0$ exists since $\htte(t,T)<\alpha$ and $\{\alpha_{\gamma}\}_{\gamma<\omega_1}$ is a continuous increasing transfinite sequence converging to $\alpha$. Since the transfinite sequence $\{\alpha_{\gamma}\}_{\gamma<\omega_1}$ is continuous and $t$ belongs to a successor level, we have that $\gamma_0\in I(\omega_1)$.\\
Suppose that $U_p\subset T_{\alpha_{\xi}+1}\setminus T_{\alpha_{\xi-1}}$ for some $\xi\in I(\omega_1)$. Let us consider three cases. If $\xi\geq \gamma_0 +1$, it follows that $t\notin \widetilde{U}_p$. If $\xi<\gamma_0$, since $\widetilde{U}_p=U_p\cup (\cup_{x\in \Lev_{\alpha_{\xi}+1}(T)\cap U}V_x)$, it follows that $t\in \widetilde{U}_p$ if and only if $\hat{t}\cap \Lev_{\alpha_{\xi}+1}(T)\subset U_p$. Finally, if $\xi=\gamma_0$, we have $t\in \widetilde{U}_p$ if and only if $t\in U_p$.\\
Hence, since $\mathcal{U}_{\xi}$ is point-countable on $D_{\xi}\subset T_{\alpha_{\xi}+1}$ for each $\xi<\gamma_0 +1$, we have that $\mathcal{U}(t)$ is countable as well.
\item Let $t \in \Lev_{\alpha+1}(T)$. Then there exists $s\in R\cap \Lev_{\alpha}(T)$ such that $t\in\ims(s)$. There exists $\gamma<\omega_1$ such that $s\in A_{\gamma}$. Take any element $X$ of the family $\mathcal{U}$ containing $t$, then there are the following possibilities:
\begin{itemize}
\item $X=\{t\}$, exactly one element of $\mathcal{U}$ has this form;
\item $X\neq \{t\}$. It follows from $(c)$ and $(d)$ that the restriction of the function $\phi$ on $\{u\in T:u<s\}$ is non-decreasing and unbounded. Let $u_0\in T$ be the minimal element $u<s$ such that $\phi(u)\geq\gamma$. Let $\xi_0<\omega_1$ be the minimal ordinal $\xi<\omega_1$ such that $u_0\in T_{\alpha_{\xi}+1}$.\\
If $\eta\in I(\omega_1)$ and $p\in \Lev_{\alpha_{\eta-1}+1}(T)$ such that $X=\widetilde{U}_p$  for some $U_p\subset T_{\alpha_{\eta}+1}\setminus T_{\alpha_{\eta-1}}$, then, since $\widetilde{U}_p\cap \bigcup_{\zeta<\phi(\min(U_p))} A_{\zeta}=\emptyset$ and $s\in A_{\gamma}$, we obtain $\phi(\min(U_p))\leq\gamma$ and therefore $\eta\leq \xi_0 +1$. Similarly, if $X=\widetilde{U}_p\setminus \ims(u)$ for some $u\in R$, then, since $(\widetilde{U}_p\setminus\ims(u))\cap \bigcup_{\zeta<\phi(\min(U_p))}\bigcup_{r\in A_{\zeta}} \ims(r)=\emptyset$ and $t\in X$, we obtain $\eta\leq \xi_0 +1$.\\
Since $\widetilde{U}_p=U_p\cup (\cup_{x\in \Lev_{\alpha_{\xi}+1}(T)\cap U}V_x)$, it follows that $t\in \widetilde{U}_p$ $(t\in \widetilde{U}_p\setminus\ims(u))$ if and only if $\hat{t}\cap \Lev_{\alpha_{\xi}+1}(T)\subset U_p$.\\
Hence, since $\mathcal{U}_{\xi}$ is point-countable on $D_{\xi}\subset T_{\alpha_{\xi}+1}$, for each $\xi\leq\xi_0+1$ we have that $\mathcal{U}(t)$ is countable as well.
\end{itemize}
\end{itemize}
Therefore $T$ is Valdivia. This concludes the proof.
\end{proof}

\section{Banach spaces of continuous functions on trees}

In this section we deal with the space of continuous functions on a tree $T$. We will prove that Valdivia compact trees can be characterized by their space of continuous functions, as we will prove that $T$ is Valdivia if and only if $C(T)$ is $1$-Plichko ($T$ has a retractional skeleton if and only if $C(T)$ has a $1$-projectional skeleton). Notice that, in general, this is not true: there are examples of a non Valdivia compacta $K$ such that $C(K)$ is $1$-Plichko see \cite{BanaKu}, \cite{Kalenda3}, \cite{KubisUspe}.\\
In the final part of this section, we will prove that each $C(T)$ space, where $T$ is a tree with height less than $\omega_1\cdot \omega_0$, is a Plichko space. Using this result, we observe that, the tree $T$ defined as in \cite[Example 4.3]{SomaTree}, is an example of compact space with retractional skeletons none of which is commutative, however $C(T)$ is a Plichko space, therefore it has a commutative projectional skeleton (see \cite[Theorem 27]{Kubis1}).

\begin{thm}\label{1PlichkoValdivia}
Let $T$ be a tree. Then $T$ is a Valdivia compact space if and only if $C(T)$ is a $1$-Plichko space.
\end{thm}

\begin{proof}
The "only if" part is a particular case of \cite[Theorem 5.2]{Kalenda2}. Suppose that $C(T)$ is a $1$-Plichko space and let $S\subset C(T)^*$ be a $1$-norming $\Sigma$-subspace. The compact space $T$ embeds canonically into $B_{C(T)^*}$ by identifying each $t\in T$ with the Dirac measure concentrated on $\{t\}$. This embedding will be denoted by $\delta$. We are going to prove that $\delta(t)=\delta_t\in S$ whenever $t\in T$ is on a successor level.\\
Pick $t\in T$ on a successor level:

\begin{itemize}
\item suppose that $\ims(t)$ is finite, it follows that $t$ is isolated. Hence, since $S$ is $1$-norming, we obtain $\delta_t\in S$;
\item on the other hand suppose that $\ims(t)$ is infinite. Let $\{t_n\}_{n\in\omega_0}$ be an infinite subset of $\ims(t)$. Since $V_{t_{n}}$ is a clopen subset of $T$ we have that the function $f_n=1_{V_{t_n}}$ is continuous for every $n\in\omega_0$. Since $S$ is a $1$-norming subset of $C(T)^*$, for each $k\in\omega_0$ there exists $\mu_{n}^{k}\in C(T)^*$ with $\|\mu_{n}^{k}\|=1$ and $\mu_{n}^{k}(f_n)>1-1/k$. Since $S$ is a $\Sigma$-subspace, by \cite[Lemma 1.6]{Kalenda2}, we have that the closure of the sequence $\{\mu_{n}^{k}\}_{k\in\omega_0}$ is contained in $S$ and moreover there exist a measure $\mu_n$ contained in the closure of $\{\mu_{n}^{k}\}_{k\in\omega_0}$ and a sequence $\{\mu_{n}^{k_j}\}_{j\in \omega_0}\subset \{\mu_{n}^{k}\}_{k\in\omega_0}$ such that $\mu_{n}^{k_j}$ converges to $\mu_n$. Hence we obtain $\mu_{n}(f_n)=\mu_n(V_{t_{n}})=1$. Observing that $\|\mu_n\|=|\mu_n|(T)\leq 1$ and $1=|\mu_n(V_{t_n})|\leq |\mu_n|(V_{t_n})$, we easily obtain that $\supp(\mu_n)\subset V_{t_n}$.\\
Now, for $f\in C(T)$ and $\varepsilon>0$ we define 
\begin{equation*}
Z_{\varepsilon}(f,t)=\{s\in\ims(t):\sup_{p\in V_s}|f(p)-f(t)|>\varepsilon\}.
\end{equation*}
By the continuity of $f$, the set $Z_{\varepsilon}(f,t)$ is finite. Hence there exists $n_0\in \omega_0$ such that 
$$\sup_{p\in V_{t_n}}|f(p)-f(t)|<\varepsilon$$
for every $n\geq n_0$. Suppose that $n\geq n_0$, then we obtain:
\begin{equation*}
\begin{split}
|\mu_{n}(f)-f(t)|&=\left|\int_{V_{t_n}}f(x)d\mu_{n}(x)-f(t)\right|=\left|\int_{V_{t_n}}f(x)-f(t)d\mu_{n}(x)\right|\\
&\leq \int_{V_{t_n}}|f(x)-f(t)|d\mu_{n}(x)<\varepsilon.
\end{split}
\end{equation*}
Hence the sequence $\mu_{n}(f)$ converges to $\delta_t(f)$ for every $f\in C(T)$. By the weak$^*$ countably closedness of $S$ it follows that $\delta_t\in S$.
\end{itemize}

Therefore, since by \cite[Theorem 5.2]{Kalenda2} $B_{C(T)^*}$ is a Valdivia compact space with $B_{C(T)^*}\cap S$ as $\Sigma$-subset and $S\cap \delta(T)$ is dense in $\delta(T)$, we obtain that $\delta(T)$ is a Valdivia compact space. 
\end{proof}

We observe that the same result can be done in the non-commutative setting. Using \cite[Proposition 3.15]{Cuth1} instead of \cite[Theorem 5.2]{Kalenda2} we obtain the following result.

\begin{thm}
Let $T$ be a tree. Then $T$ has a retractional skeleton if and only if $C(T)$ has a $1$-projectional skeleton.
\end{thm}

Now we are going to investigate the space of continuous function on trees with height less than $\omega_1\cdot \omega_0$. It turns out that all such spaces are Plichko. 
 
\begin{thm}\label{PlichkoTree}
Let $T$ be a tree such that $\htte(T)<\omega_1\cdot\omega_0$, then $C(T)$ is a Plichko space.
\end{thm}

The previous theorem follows immediately from the next technical proposition, where, for every tree $T$ of height less than $\omega_1\cdot\omega_0$, a norming $\Sigma$-subspace of $C(T)^*$ is explicitly described.

\begin{prop}\label{PropPlichkoTree}
Let $T$ be a tree and suppose that $\htte(T)\leq (\omega_1\cdot n) + 1$, for some $n \geq 1$. Then 
\begin{equation*}
\Lambda=\{\mu\in C(T)^*: \forall j\in \{1,...,n\},\forall t\in \Lev_{\omega_1\cdot j}(T):\, \mu(V_t)=0\}
\end{equation*}
is a $(2n-1)$-norming $\Sigma$-subspace of $C(T)^*$. If $\htte(T)>(\omega_1\cdot(n-1))+1$, then the norming constant is exactly $2n-1$.
\end{prop}

\begin{lem}\label{RepresentationSigmaSub}
Let $T$ be a tree such that $\htte(T)\leq \omega_1 +1$ and $D=\{t\in T: \cf(t)\leq\omega_0\}$. Then the set
\begin{equation*}
S=\{\mu\in C(T)^*:\supp(\mu)\subset D\}
\end{equation*}
is a $1$-norming $\Sigma$-subspace of $C(T)^*$.
\end{lem}

\begin{proof}
Since $\htte(T)\leq\omega_1+1$, by \cite[Theorem 4.1]{SomaTree}, $T$ is a Valdivia compact space and $D$ is a dense $\Sigma$-subspace. Hence, by \cite[Proposition 5.1]{Kalenda2} we have that
\begin{equation*}
S=\{\mu\in C(T)^*:\supp(\mu) \mbox{ is a separable subset of } D\}
\end{equation*}
is a $1$-norming $\Sigma$-subspace of $C(T)^*$. Finally the assertion follows by Proposition \ref{Radonseparablesupport}.
\end{proof}

\begin{proof}[Proof of Proposition \ref{PropPlichkoTree}]
If $n=1$ the assertion follows from Lemma \ref{RepresentationSigmaSub}, hence we assume that $n\geq 2$ and that $T$ is a tree with $\omega_1\cdot (n-1)+1<\htte(T)\leq (\omega_1\cdot n) +1$. As in Lemma \ref{RepresentationSigmaSub} we define $D=\{t\in T: \cf(t)\leq \omega_0\}$. Let $S_0=\emptyset$, $S_i=T_{\omega_1\cdot i}$ for each $i\leq n-1$ and $S_{n}=T$. Then we obtain the following:
\begin{itemize}
\item $S_i$ is a closed subset of $T$ for every $i\in\{1,...,n\}$;
\item $S_1$ is isomorphic to a tree of height $\omega_1 +1$, hence, by Lemma \ref{RepresentationSigmaSub}, $C(S_1)$ is a $1$-Plichko space with $\Sigma_1=\{\mu\in C(S_1)^*:\supp(\mu)\subset D\cap S_1\}$ as $\Sigma$-subspace;
\item for every $i\in\{1,...,n-1\}$, the subset $S_{i+1}\setminus S_i$ is a locally compact space and $C_0(S_{i+1}\setminus S_i)$ is a $1$-Plichko space. Indeed, let $t\in \Lev_{(\omega_1\cdot i)+1}(T)$ and $U_{t}=V_{t}\cap S_{i+1}$. It is clear that $U_t$ is a closed subset of $T$ and further it is isomorphic to a tree of height less or equal than $\omega_1 +1$. Hence, by Lemma \ref{RepresentationSigmaSub}, we obtain that $U_t$ is a Valdivia compact space and $C(U_t)$ is a $1$-Plichko space with $\Sigma_{i,t}=\{\mu\in C(U_t)^*:\supp(\mu)\subset D\cap U_t\}$ as $\Sigma$-subspace. Moreover we observe that $S_{i+1}\setminus S_i$ is the topological sum of all $U_t$, therefore $C_{0}(S_{i+1}\setminus S_i)$ is the $c_0$-sum of $C(U_t)$ and its dual is the $\ell_1$-sum of $C(U_t)^*$. Hence, by \cite[Theorem 4.31 and Lemma 4.34]{Kalenda2} we obtain that $C_0(S_{i+1}\setminus S_i)$ is a $1$-Plichko space and 
\begin{equation*}
\begin{split}
\Sigma_{i}=&\{(\mu_t)_{t\in \Lev_{(\omega_1\cdot i)+1}(T)}\in C_0(S_{i+1}\setminus S_i)^*:(\forall t\in  \Lev_{(\omega_1\cdot i)+1}(T))(\mu_t\in \Sigma_{i,t})\, \& \\
&\{t\in \Lev_{(\omega_1\cdot i)+1}(T):\mu_t\neq 0\} \mbox{ is countable}\}
 \end{split}
\end{equation*}
 is its $\Sigma$-subspace.
\end{itemize}

Suppose that $i\leq n-1$ and let $r_{i}:T\to T$ be the continuous retraction defined by:
\begin{equation*}
r_i(t)=
\begin{cases}
t&\mbox{  if } t\in S_i,\\
s&\mbox{ if } s\leq t \mbox{ and } s\in \Lev_{\omega_1\cdot i}(T).
\end{cases}
\end{equation*}
For simplicity we define $r_{n}:T\to T$ as the identity map. These continuous retractions induce continuous linear projections on $C(T)$ and such projections are defined by $P_i(f)=f\circ r_i$. Then for every $f\in C(T)$ and every $i\leq n-1$ the following conditions hold:
\begin{itemize} 
\item $f\upharpoonright_{S_i}=P_i f\upharpoonright_{S_i}$;
\item $(f-P_i f)\upharpoonright_{S_{i+1}}=(P_{i+1}f-P_{i}f)\upharpoonright_{S_{i+1}}$; 
\item $(f-P_i f)\upharpoonright_{S_{i+1}\setminus S_i}\in C_0(S_{i+1}\setminus S_i)$.
\end{itemize}

In order to get an isomorphism between $C(T)$ and a $1$-Plichko space we are going to define the following map:
\begin{equation*}
\begin{split}
G:C(&T)\to C(S_1)\oplus_{\infty}C_0(S_{2}\setminus S_1)\oplus_{\infty}...\oplus_{\infty}C_0(S_{n}\setminus S_{n-1})\\
&f\mapsto(f\upharpoonright_{S_1},(f-P_1 f)\upharpoonright_{S_2\setminus S_1},...,(f-P_{n-1} f)\upharpoonright_{S_{n}\setminus S_{n-1}}).
\end{split}
\end{equation*}
This easily implies that the norm of $G$ is at most $2$. Now we are going to define the inverse of the mapping $G$. For simplicity we denote by \begin{equation*}
W=C(S_1)\oplus_{\infty}C_0(S_{2}\setminus S_1)\oplus_{\infty}...\oplus_{\infty}C_0(S_{n}\setminus S_{n-1}).
\end{equation*}
Let $(f_1,f_2,...,f_{n})\in W$. We define its preimage $f\in C(T)$ as follows:
\begin{equation*}
f(t)=
\begin{cases}
f_1(t)&\mbox{  if } t\in S_1,\\
f_{i+1}(t) + \sum_{j=1}^{i}f_j(r_j(t))&\mbox{ if } t\in S_{i+1}\setminus S_i.
\end{cases}
\end{equation*}
It follows that the norm of the inverse of $G$ is at most $n$. Therefore $G$ is an isomorphism. Since each component of $W$ is a $1$-Plichko space, we have that $W$ is $1$-Plichko, therefore $C(T)$ is a Plichko space. Moreover, the subspace $\Sigma=\{(\mu_i)_{i=1}^{n}\in W^*:\mu_i\in \Sigma_i\}\subset W^*$ is a $1$-norming $\Sigma$-subspace of $W$. In order to compute the exact value of the norming constant of the $\Sigma$-subspace $G^{*}(\Sigma)$, we are going to describe the adjoint map of $G$:
\begin{equation*}
\begin{split}
G^*(\mu_1,\dots,\mu_{n})(f)&=(\mu_1,\dots,\mu_{n})(Gf)=
\mu_1(f\upharpoonright_{S_1}) +\sum_{j=1}^{n-1} \mu_{j+1}((f-P_jf)\upharpoonright_{S_{j+1}\setminus S_j})\\
&=\int_{S_1}f d\mu_1 + \sum_{j=1}^{n-1}\int_{S_{j+1}\setminus S_j}(f-P_jf)d \mu_{j+1}\\
&=\int_{T}f d(\sum_{i}^{n}\mu_i) - \sum_{j=1}^{n-1}\int_{T}f d r_j(\mu_{j+1}),\\
\end{split}
\end{equation*}
Hence 
\begin{equation*}
G^*(\mu_1,\dots,\mu_{n})=
\sum_{i=1}^{n}\mu_i -\sum_{j=1}^{n-1}r_j(\mu_{j+1}).
\end{equation*}
Where $r_i(\mu_j)(A)=\mu_j(r_{i}^{-1}(A))$ for every measurable subset $A$ of $T$.  Now we are going to give a representation of the inverse of $G^*$. Let $\mu=G^*(\mu_1,...,\mu_n)=\sum_{i=1}^{n}\mu_i -\sum_{j=1}^{n-1}r_j(\mu_{j+1})$ and $k\leq n-1$, then
\begin{equation*}
r_{k}(\mu)=\sum_{i=1}^{n}r_k(\mu_i) -\sum_{j=1}^{n-1}r_k(r_j(\mu_{j+1})).
\end{equation*}
Further we observe that $r_k(\mu_i)=\mu_i$ for $i\leq k$ and
\begin{equation*}
r_k(r_j(\mu_{j+1}))=
\begin{cases}
r_j(\mu_{j+1}), &\mbox{ if } j< k,\\
r_k(\mu_{j+1}), &\mbox{ if } j\geq k.
\end{cases}
\end{equation*}
Hence we obtain
\begin{equation*}
\begin{split}
r_{k}(\mu)&=\sum_{i=1}^k\mu_i+\sum_{i=k+1}^n r_k(\mu_i)-\sum_{j=1}^{k-1}r_j(\mu_{j+1})-\sum_{j=k}^{n-1}r_k(\mu_{j+1})\\
&=\sum_{i=1}^{k}\mu_i -\sum_{j=1}^{k-1}r_j(\mu_{j+1}).
\end{split}
\end{equation*}
Now we take the restriction of $\mu$ to $S_{i}\setminus S_{i-1}$:
\begin{equation*}
\begin{split}
&\mu\upharpoonright_{S_1}=\mu_1 -r_1(\mu_2),\\
&\mu\upharpoonright_{S_i\setminus S_{i-1}}=\mu_i-r_{i}(\mu_{i+1}), \mbox{ for }i\in \{2,...,n-1\},\\
&\mu\upharpoonright_{T\setminus S_{n-1}}=\mu_n.
\end{split}
\end{equation*}
Hence, combining these two formulae we obtain
\begin{equation*}
\begin{split}
&\mu_1=r_1(\mu),\\
&\mu_i= r_i(\mu) -\mu\upharpoonright_{S_{i-1}}, \mbox{ for }i\in\{2,...,n-1\},\\
&\mu_{n}=\mu\upharpoonright_{T\setminus S_{n-1}}.
\end{split}
\end{equation*}
Therefore  the inverse of $G^*$ can be represented as
\begin{equation*}
\mu\mapsto (r_1(\mu),r_2(\mu)-\mu\upharpoonright_{S_1},\dots,r_{n-1}(\mu)-\mu\upharpoonright_{S_{n-2}},\mu\upharpoonright_{T\setminus S_{n-1}}).
\end{equation*}
Hence we have the following:
\begin{equation*}
\begin{split}
G^*(\Sigma)&=\{\mu\in C(T)^*:(r_1(\mu),r_2(\mu)-\mu\upharpoonright_{S_1},\dots,r_{n-1}(\mu)-\mu\upharpoonright_{S_{n-2}},\mu\upharpoonright_{T\setminus S_{n-1}})\in\Sigma\}\\
&=\{\mu\in C(T)^*: \forall j\in\{1,\dots,n\}, \forall B\subset \Lev_{\omega_1\cdot j}(T): \, \mu(\bigcup_{t\in B} V_t)=0\}\\
&=\{\mu\in C(T)^*: \forall j\in\{1,\dots,n\}, \forall t\in \Lev_{\omega_1\cdot j}(T): \, \mu(V_t)=0\}.
\end{split}
\end{equation*}
Indeed, the first equality is obvious. Let us prove the second one:
\begin{itemize}
\item[$\subset$]: Let $\mu\in G^*(\Sigma)$ and $B\subset \Lev_{\omega_1\cdot j}(T)$ for some $j\in \{1,...,n-1\}$. Then, since $(r_{j}(\mu)-\mu\upharpoonright_{S_{j-1}})\in \Sigma_j$ we have $(r_j(\mu)-\mu\upharpoonright_{S_{j-1}})(B)=0$. Hence $0=(r_j(\mu)-\mu\upharpoonright_{S_{j-1}})(B)=\mu(\bigcup_{t\in B}V_t)-\mu(B\cap S_{j-1})=\mu(\bigcup_{t\in B}V_t)$. If $j=n$ we have $\mu\upharpoonright_{T\setminus S_{n-1}}\in \Sigma_n$, hence $0=\mu\upharpoonright_{T\setminus S_{n-1}}(B)=\mu(B)$.
\item[$\supset$]: Let $\mu\in C(T)^{*}$ such that for each $j\in\{1,..,n\}$ and each $B\subset \Lev_{\omega_1\cdot j}(T)$, $\mu(\bigcup_{t\in B} V_t)=0$ holds. We observe that, by Proposition \ref{Radonseparablesupport}, the support of the measure $r_{j}(\mu)-\mu\upharpoonright_{S_{j-1}}$ is a metrizable subset of $T$ for each $j\in\{2,...,n-1\}$. Hence $(r_{j}(\mu)-\mu\upharpoonright_{S_{j-1}})(V_t)=0$ for all but countably many $t\in \Lev_{(\omega_1\cdot j)+1}(T)$. Suppose that $j\in\{2,...,n-1\}$ and let $s$ be an element of $T$ on a successor level such that $\omega_1\cdot (j-1)<\htte(s,T)\leq \omega_1\cdot (j+1)$, then we have: 
\begin{equation*}
\begin{split}
(r_j(\mu)-\mu\upharpoonright_{S_{j-1}})(V_s)&=r_j(\mu)(V_s)=\mu(V_s)=\mu(V_s\cap T_{<\omega_1\cdot j})+\mu((V_s\cap (T\setminus T_{<\omega_1\cdot j}))\\
&=\mu(V_s\cap T_{<\omega_1\cdot j}).
\end{split}
\end{equation*}
In particular, for each $s\in \Lev_{(\omega_1\cdot (j-1))+1}(T)$, we have $(r_j(\mu)-\mu\upharpoonright_{S_{j-1}})(V_s)=\mu(V_s\cap T_{<\omega_1\cdot j})$, hence $(r_j(\mu)-\mu\upharpoonright_{S_{j-1}})\upharpoonright_{V_s}\subset V_s\cap D$. Therefore we obtain $(r_j(\mu)-\mu\upharpoonright_{S_{j-1}})\in\Sigma_
j$ for each $j\in\{2,...,n-1\}$. Using a similar argument we obtain $r_{1}(\mu)\in\Sigma_1$ and $\mu\upharpoonright_{T\setminus S_{n-1}}\in \Sigma_n$.
\end{itemize}
Let us prove the last one:
\begin{itemize}
\item[$\subset$]: is trivial.
\item[$\supset$]: Let $\mu\in C(T)^*$ such that for any $j\in \{1,...,n\}$ and $t\in\Lev_{\omega_1\cdot j}(T)$, $\mu(V_t)=0$ holds. Fix $j\in \{1,...,n\}$. Let us define the continuous part of $\mu$ by $\mu_c$, by Corollary \ref{supportomisure}, $\supp(\mu_c\upharpoonright S_j)$ is contained in $T_{\alpha}$, where $\alpha<\omega_1\cdot j$. Hence we may suppose that $\mu(V_t\cap S_j)=0$ whenever $\htte(t,T)>\alpha$. Hence for each relatively open subset $A\in\Lev_{\omega_1\cdot j}(T)$ we have $\mu(A)=0$, therefore, by the regularity of $\mu$, we have the same conclusion for each subset of $\Lev_{\omega_1\cdot j}(T)$.  
\end{itemize}
Hence $\Lambda=G^*(\Sigma)$ is a $\Sigma$-subspace of $C(T)^*$. Now are going to show that the norming constant of $\Lambda$ is equal to $2n-1$. At first we observe that if $f\in C(T)$ and $t\in\Lev_{\alpha}(T)$ with $\cf(t)=\omega_1$, then there exists $s\leq t$ on a successor level, such that $f$ is constant on $V_s\cap T_{\alpha}$. Indeed, since $f$ is a continuous function on $T_{\alpha}$, for each $n\in\omega_0$ there exists $t_n< t$ on a successor level such that $|f(t)-f(s)|<1/n$, for each $s\in V_{t_n}\cap T_{\alpha}$. Then we have $f(s)=f(t)$ for each $s\in V_{t_0}\cap T_{\alpha}$, where $t_0=(\sup_{n\in\omega_0}t_n)+1$.\\
Now, let $f\in C(T)$, without loss of generality, we may suppose $\|f\|=1$. Let $t\in T$ such that $|f(t)|=1$, then there exists $i\in\{0,...,n-1\}$ such that $\omega_1\cdot i\leq \htte(t,T)<\omega_1\cdot (i+1)$. We will show that there exists a measure $\mu\in \Lambda$ satisfying $\|\mu\|\leq 2n-1$ and $\mu(f)=1$, for this purpose let us consider $\mu\in \Lambda$ defined by:
\begin{equation*}
\mu=(\sum_{k=1}^{i}\delta_{t_k}-\delta_{s_{k}}) +\delta_{t_0},
\end{equation*}
where $t_i=t$, $\{s_{k}\}=\hat{t}_i\cap \Lev_{\omega_1\cdot k}(T)$ and $t_k$ is such that $s_{k}<t_k<s_{k+1}$ ($t_0<s_1$) and $f(s_{k+1})=f(t_k)$ for $k\geq 1$ (for $k=0$, respectively). Such elements exist since $f$ is constant near points of uncountable cofinality. Therefore we obtain $2n-1\geq 2i+1=\|\mu\|$ and easily we get $|\mu(f)|=1$. Hence the norming constant of $\Lambda$ is at most $2n-1$. On the other hand, suppose that $\htte(T)>(\omega_1\cdot (n-1))+1$ and let $t_{n-1}\in T$ such that $\htte(t_{n-1},T)>(\omega_1\cdot (n-1))+1$ and 
\begin{equation*}
\begin{split}
&\{s_i\}=\hat{t}_{n-1}\cap \Lev_{\omega_1\cdot i}(T) \mbox{ for }i=1,...,n-1,\\
&\{t_{i}\}=\hat{t}_{n-1}\cap \Lev_{(\omega_1\cdot i)+1}(T)\mbox{ for }i=1,...,n-2.
\end{split}
\end{equation*}
Let us consider the following continuous map:
\begin{equation*}
f(t)=1_{V_{t_{n-1}}}(t)+ \delta_1 1_{T\setminus V_{t_1}}(t) +\sum_{i=1}^{n-2}\delta_{i+1}1_{V_{t_i}\setminus V_{t_{i+1}}}(t),
\end{equation*}
where $\delta_1=1/(2n-1)$ and $\delta_i=(2i-1) \delta_1$ for $i=1,..,n-1$. Let $\mu\in \Lambda$ such that $\mu(f)=1$. We put
\begin{equation*}
\begin{split}
&\mu(T\setminus V_{s_1})=a_0,\\
&\mu(V_{s_i}\setminus V_{t_i})=b_i \mbox{  for } i=1,...,n-1,\\
&\mu(V_{t_i}\setminus V_{s_{i+1}})=a_i \mbox{  for } i=1,...,n-2,\\
&\mu(V_{t_{n-1}})=a_{n-1}
\end{split}
\end{equation*}
Since $\mu\in \Lambda$ we have $b_i=-a_i$ for every $i=1,...,n-1$. Therefore we obtain:
\begin{equation*}
\begin{split}
1&=\mu(f)= a_{n-1}+ \sum_{i=1}^{n-1} \delta_{i}(a_{i-1}-a_i)\\
&= \delta_1 a_0 + (1-\delta_{n-1})a_{n-1} + \sum_{i=1}^{n-2}a_i(\delta_{i+1}-\delta_i)\\
&\leq (|a_0|+\sum_{i=1}^{n-1}2|a_i|)\cdot \max\{\delta_1,\frac{1-\delta_{n-1}}{2}\}\\
&=(|a_0|+\sum_{i=1}^{n-1}2|a_i|)\cdot \frac{1}{2n-1}.
\end{split}
\end{equation*}
Hence $\|\mu\|\geq 2n-1$. This concludes the proof.
\end{proof}

Combining Theorem \ref{caratValdiviaGeneral} and Theorem \ref{1PlichkoValdivia} we have several examples of trees $T$, also with height bigger than $\omega_1\cdot \omega_0$, such that $C(T)$ is a $1$-Plichko space. However, in the final part of the proof of Theorem \ref{PlichkoTree}, the norming constant of that particular $\Sigma$-subspace grows as $2n-1$. This means that, in general, this is not the optimal choice for the $\Sigma$-subspace. This fact naturally rises the following question.
\begin{problem}
 Let $T$ be a tree with height equal to $\omega_1\cdot \omega_0$. Is $C(T)$ necessarily Plichko?
\end{problem}

\begin{remark}
We have assumed that every tree was rooted, the above results can be proved also if the tree $T$ has finitely many minimal elements. Indeed, if $T$ has finitely many minimal elements, then it can be viewed as the topological direct sum of rooted trees.
\end{remark}

\textbf{Acknowledgements:} The author is grateful to Ond\v{r}ej Kalenda for many helpful discussions.

Jacopo Somaglia\\

Dipartimento di Matematica\\
Universit\`{a} degli studi di Milano\\
Via C. Saldini, 50\\
20133 Milano MI, Italy\\
\\
Department of Mathematical Analysis,\\
Faculty of Mathematics and Physics,\\
Charles University,\\ 
Sokolovsk\'{a} 83,\\
186 75 Praha 8, Czech Republic\\
\\
jacopo.somaglia@unimi.it\\

\end{document}